\documentclass[11pt, a4paper, twoside]{amsart}

\usepackage{mathrsfs, tabularx}
\usepackage{amsthm, amsmath, amssymb, amscd, float, latexsym, times, epic, eepic}
\usepackage{graphicx,epsfig}

\newtheorem {lemma}{Lemma}\newtheorem {proposition}{Proposition}
\newtheorem {theorem}{Theorem}
\newtheorem {definition}{Definition}
\newtheorem {example}{Example}
\newtheorem {corollary}{Corollary}

\newtheorem {remark}{Remark}

\def \N {{\mathbb N}}
\def \R {{\mathbb R}}

\def\bra{\langle}
\def\cet{\rangle}

\newcommand{\norm}[1]{\left\|#1 \right\|}

\newcommand{\T}{\mathbb{T}}

\def\expec{\mathbb{E}}

\DeclareMathOperator*{\argmin}{arg\,min}

\newcommand{\map}{u_{MAP}}
\newcommand{\cm}{u_{CM}}

\def \p{\partial}
\def \sign{{\rm sign}}
\def \D{{\widetilde D}}

\newcommand{\correction}[1]{#1}

\title[MAP estimates in infinite-dimensional Bayesian inverse problems]{Maximum a posteriori probability estimates in infinite-dimensional Bayesian inverse problems}

\author{Tapio Helin$^*$ and Martin Burger$^\circ$}

\date{\noindent \today}
\thanks{\noindent $^*$ Department of Mathematics and Statistics, University of Helsinki. tapio.helin@helsinki.fi 
\\
$^\circ$ Institute for Computational and Applied Mathematics, Westf\"alische Wilhelms-Universit\"at M\"unster, and Cells in Motion Cluster of Excellence, M\"unster. martin.burger@wwu.de}

\begin{document}

\begin{abstract}
A demanding challenge in Bayesian inversion is to efficiently characterize the posterior distribution.
This task is problematic especially in high-dimensional non-Gaussian problems, where the structure of
the posterior can be very chaotic and difficult to analyse.
Current inverse problem literature often approaches the problem by considering suitable point estimators
for the task. Typically the choice is made between the maximum a posteriori (MAP) or the conditional mean (CM)
estimate. 

The benefits of either choice are not well-understood from the perspective of infinite-dimensional theory.
Most importantly, there exists no general scheme regarding how to connect the topological description of a MAP estimate to a variational problem. The recent results by Dashti and others \cite{Dashti13} resolve this issue for non-linear inverse problems in Gaussian framework.
In this work we improve the current understanding by introducing a novel concept called the weak MAP (wMAP) estimate.
We show that any MAP estimate in the sense of \cite{Dashti13} is a wMAP estimate and, moreover, how the wMAP estimate connects to a variational formulation in general infinite-dimensional non-Gaussian problems. The variational formulation enables to study many properties of the infinite-dimensional MAP estimate that were earlier impossible to study.

In a recent work by the authors \cite{BL14} the MAP estimator was studied in the context of the Bayes cost method.
Using Bregman distances, proper convex Bayes cost functions were introduced for which the MAP estimator is the Bayes estimator.
Here, we generalize these results to the infinite-dimensional setting. Moreover, we discuss the implications of our results for some examples of prior models such as the Besov prior and hierarchical prior.
\end{abstract}

\maketitle

\tableofcontents

\section{Introduction}

Bayesian inversion recasts inverse problems in the form of a statistical quest for information.
From the Bayesian perspective, the solution to an inverse problem is the probability distribution of the unknown when all information available has been incorporated into the model \cite{KS}. 
This solution, called the \emph{posterior distribution}, describes our best understanding of what are the more and less probable values of the unknown.

The drawback of the Bayesian method, especially for high-dimensional problems, is the challenge to represent and process the information encoded in the posterior. What is a good representative of the posterior distribution?
Moreover, how does this information change if the discretization of the problem is refined?

A widely used approach in the inverse problem literature is to consider either the maximum a posteriori (MAP) or the conditional mean (CM) estimate as the ultimate representative. However, the topic is much debated and it is currently unclear under what conditions a high-dimensional non-Gaussian posterior distribution is well characterized by either estimator \cite{LS04}.

In the scheme of Bayes cost formalism, the MAP estimate is often discredited for being only asymptotically a Bayes estimator
for the uniform cost function.
Recent work by the authors \cite{BL14} sheds new light on this topic by introducing proper convex Bayes cost functions for which the MAP estimator is the Bayes estimator. This result, utilizing so-called Bregman distances, indicates that the MAP estimate can provide better representation properties for common non-Gaussian posteriors originating e.g. from sparsity priors \cite{KLNS12}.
Unfortunately, the techniques used in \cite{BL14} are limited to finite-dimensional problems. The infinite-dimensional generalization is challenging due to the fact that it is not well-understood how the definition of a MAP estimate connects to the standard variational formulation in non-Gaussian problems. 

The lack of infinite-dimensional theory on the MAP estimate is problematic also from another perspective. Namely,
it is currently not known when the MAP estimates are discretization invariant \cite{LSS09}. In other words,
if the discretization of the unknown is refined, does the finite-dimensional MAP estimates converge to the infinite-dimensional counter-part? \correction{The answer is known to be negative in some cases \cite{LS04}. From practical point of view, it would be important to distinguish and better understand such cases.} A variational characterization of the limiting infinite-dimensional estimate could provide more insight into this problem.

For Gaussian prior and noise distributions this issue was solved in \cite{Dashti13} by utilizing the theory of small ball probabilities. In this paper, we improve the situation in the non-Gaussian setting by introducing new means to characterize the MAP estimate. This approach stems from differentiability calculus developed for measures by Fomin in late 1960s \cite{Fomin68,Fomin70}. In the core of our method is an interesting connection between the differentiability and quasi-invariance of measures originally discovered by Skorohod (see \cite{Skorohod}).
This theory relates so-called logarithmic derivative of probability measures and the generalized Onsager--Machlup functional. Here, such a connection is fundamental since the logarithmic derivative is directly connected to the variational formalism via its zero points (interpreted in suitable sense), whereas the Onsager--Machlup functional is related to the topological definition introduced in \cite{Dashti13}.
With the help of these tools we are able to introduce a weak formulation of the MAP estimate that builds
the much needed bridge between the two formalisms. Moreover, our work generalizes the results shown in \cite{BL14} related to the Bayes cost method. We want to point out that the possibility of studying zero points of the logarithmic derivative has already been discussed in \cite{Hegland}. 

Let us mention that the asymptotics of the posterior was first considered using total variation (TV) priors in \cite{LS04}. This inspirational paper illustrated how the TV prior can asymptotically lose its edge-preserving property. Likewise, the conditions for convergence regarding the MAP and CM estimates were shown to be inconsistent. It was in this paper where the concept of discretization invariance was first coined for non-Gaussian priors, continuing a line of research starting in \cite{lehtinen1989linear}. Similar inconsistency was also studied in the framework of hierarchical priors in the earlier work by the authors \cite{Helin09,HL11}. The concept of discretization invariance has later been refined in \cite{Helin09,LSS09}. Notice that if the prior distribution converges weakly, the posterior can be shown to converge weakly in general settings \cite{Lasanen1,Lasanen2}.
Moreover, Bregman distance in connection to Bayesian models has been considered earlier in \cite{Frigyik08, Goh14}.
For an excellent overview of infinite-dimensional Bayesian inverse problems, see \cite{Stuart,Lasanen1}.

This paper is organized as follows. In Section \ref{sec:prelim} we discuss the problem setting and main tools regarding Fomin differential calculus. 
Section \ref{sec:onsager_and_map} covers essential results on the Onsager--Machlup functional in infinite-dimensional spaces. Moreover, a weak definition of the MAP estimate is given and its connection to earlier definition in \cite{Dashti13} is studied.
Variational characterization of the weak MAP estimate is described in Section \ref{sec:variational}, where we generalize the work in \cite{BL14}.
Finally, in Sections \ref{sec:besov} and \ref{sec:hierarchical} we illustrate what our results imply for the Besov prior and the hierarchical prior, respectively.

\section{Preliminaries}
\label{sec:prelim}
%

We consider the inverse problem of solving a linear equation 
\begin{equation}
	\label{eq:inv_problem}
	m = Au + e,
\end{equation}
for the unknown $u$ given the measurement $m\in\R^M$.
Above, the unknown $u$ belongs to a separable Banach space $X$, the operator $A : X \to \R^M$ is linear and $e \in \R^M$ models the noise. In the Bayesian paradigm the unknown in \eqref{eq:inv_problem}
is modelled by a random variable. The task is to estimate the conditional distribution
of $u$ given the measurement, i.e. a sample of $m$. Note that $m$ and $A$ are usually to be interpreted as discretizations of a random variable on an infinite-dimensional space $Y$ and a linear operator mapping to $Y$, respectively. Since we are mainly interested in aspects of priors in infinite-dimensional spaces we do not carry out the limit $M\rightarrow \infty$ in the non-Gaussian setting of this paper, but leave it as a relevant question for future research. 

It is well-known that the solution to this problem is achieved via the Bayes formula. Let us assume that $\lambda$ is the prior probability distribution of $u$ and $e$ is normally i.i.d. vector. Then the conditional distribution $\mu$ of $u$ given $m$ satisfies
\begin{equation}
	\label{eq:bayes}
	\mu({\mathcal A} \; | \; m) = \frac{1}{G} \int_{{\mathcal A}} \pi(m \; | \; u) \lambda(du),
\end{equation}
for almost every $m \in \R^M$, where ${\mathcal A} \in {\mathcal B}(X)$, $\pi(m \; | \; u)$ is the likelihood density and $G = \int_{\R^M} \pi(m \; | \; u) \lambda(du)$ is the normalizing constant \cite{Schervish}. 
\begin{remark}
We point out that in Section \ref{sec:onsager_and_map} the posterior distribution in equation \eqref{eq:bayes} is considered from general perspective. Afterwards,
for the rest of the paper we assume a Gaussian likelihood, i.e.,
\begin{equation}
	\label{eq:gaussian_like}
	\pi(m \; | \; u) \propto \exp\left(-\frac 12|Au-m|^2\right).
\end{equation}
In fact, also there the argument could be generalized since we use only properties including the differentiability, boundedness and log-concavity of $\pi(m\;|\;u)$.
However, arbitrary likelihood density would require additional consideration in the Bayes cost method in Section \ref{sec:variational} and for simplicity we restrict ourselves to this particular case.
\end{remark}

The following concept originating to papers by Fomin in the 1960s \cite{Fomin68,Fomin70} is the crux of this paper.
A good overview of the Fomin calculus is given in \cite{Boga10}.
\begin{definition}
A measure $\mu$ on $X$ is called Fomin differentiable along the vector $h$ if, for every set ${\mathcal A}\in {\mathcal B}(X)$, there exists a finite limit
\begin{equation}
	\label{eq:dhmu_def}
	d_h\mu({\mathcal A}) = \lim_{t\to 0} \frac{\mu({\mathcal A}+th)-\mu({\mathcal A})}{t}
\end{equation}
\end{definition}

The set function $d_h \mu$ defined by \eqref{eq:dhmu_def} can be written as a pointwise limit of the sequence of measures ${\mathcal A} \mapsto n\left(\mu({\mathcal A}+n^{-1}h)-\mu({\mathcal A})\right)$. Therefore,
by the Nikodym theorem it is a countably additive signed measure on ${\mathcal B}(X)$ and has bounded variation \cite{Boga10}.

We denote the domain of differentiability by
\begin{equation}
	D(\mu) = \{h \in X \; | \; \mu \textrm{ is Fomin differentiable along } h\}
\end{equation}
If $\mu$ is a probability measure and $h\in D(\mu)$, then the function $f(t) = \mu({\mathcal A} + th)$
is a non-negative differentiable function. Clearly, if $\mu({\mathcal A}) = f(0) = 0$ then we must also have $f'(0) = 0$.
In consequence, $d_h \mu$ is absolutely continuous with respect to $\mu$.

\begin{definition}
The Radon--Nikodym density of the measure $d_h\mu$ with respect to $\mu$ is denoted
by $\beta^\mu_h$ and is called the logarithmic derivative of $\mu$ along $h$.
\end{definition}

Consequently, for all ${\mathcal A}\in {\mathcal B}(X)$ the logarithmic gradient $\beta^\mu_h$ satisfies
\begin{equation}
	d_h\mu({\mathcal A}) = \int_{\mathcal A} \beta^\mu_h(u) \mu(du)
\end{equation}
and, in particular, we have $d_h\mu(X) = 0$
for any $h\in D(\mu)$ by definition. Moreover, $\beta_{sh}^\mu = s \cdot \beta_h^\mu$ for any $s\in\R$.
It turns out that the space $D(\mu)$ with the norm $\norm{h} = \norm{\beta^\mu_h}_{L^1(\mu)}$ is a Banach space compactly embedded into $X$ \cite[Thm 5.1.1.]{Boga10}. 

Let us include an equivalent formulation of Fomin differentiability. We denote by ${\mathcal FC^\infty}(X)$ the collection of all smooth  cylindrical functions $f$ on $X$, that is, $f$ is of the form
\begin{equation*}
	f(u) = \phi(\ell_1(u),...,\ell_n(u)), \quad \phi \in C^\infty_0(\R^n), \quad \ell_i \in X^*, n\in \N.
\end{equation*}
Using such test functions, Fomin differentiability can be expressed in the following weak sense.
\begin{proposition}
A Radon measure $\mu$ on $X$ is differentiable along a vector $h\in X$ in the sense of Fomin if and only if there exists a function $\beta^\mu_h \in L^1(\mu)$ such that for all $f \in {\mathcal FC^\infty}(X)$,
the following integration by parts formula holds
\begin{equation}
	\label{eq:weak_fomin}
	\int_X \partial_h f(u) \mu(du) = - \int_X f(u) \beta^\mu_h(u) \mu(du).
\end{equation}
The function $\beta^\mu_h$ is called the logarithmic derivative of the measure $\mu$ along $h$.
\end{proposition}
Notice that equation \eqref{eq:weak_fomin} is sometimes used as the definition of Fomin differentiability \cite{Boga_gaussian,BogaMayer99}. Indeed, when solving the logarithmic derivative, the weak formulation is typically
the natural approach as we will see in the next example. 

\begin{example}
\label{ex:gaussian1}
Let us consider a Gaussian measure $\gamma$ on $(X,{\mathcal B}(X))$. 
Suppose $X$ is a separable Hilbert space and let $T$ be a non-negative self-adjoint Hilbert--Schmidt operator on $X$. Let also $\gamma$ be zero-mean with a covariance operator $T^2$. Then the Cameron--Martin space of $\gamma$ is defined by
\begin{equation*}
	H(\gamma) := T(X), \quad\quad  \langle h_1,h_2\rangle_{H(\gamma)} = \langle T^{-1} h_1, T^{-1} h_2\rangle_X.
\end{equation*}
Now suppose $f\in {\mathcal FC}^\infty(X)$. The Cameron--Martin formula yields that
\begin{equation*}
	\int_X \frac{f(u+th)-f(u)}{t} \gamma(du) = \int_X f(u) \frac{r(t,u)-1}{t} \gamma(du),
\end{equation*}
where
\begin{equation*}
	r(t,u) = \exp\left(t \langle h,u\rangle_{H(\gamma)}- \frac{t^2}2 \norm{h}_{H(\gamma)}^2\right).
\end{equation*}
By taking $t\to 0$ we obtain by the Lebesgue dominated convergence theorem that
\begin{equation*}
	\int_X \partial_h f(u) \gamma(du) = \int_X f(u) \langle h,u\rangle_{H(\gamma)} \gamma(du)
\end{equation*}
for any $h \in H(\gamma)$. Moreover, if $h\notin H(\gamma)$ then it is well-known that
$\gamma$ and translated measure $\gamma(\cdot - th)$ are mutually singular for all $t>0$.
Hence, $\gamma$ is not differentiable along $h$. In consequence, we have
\begin{equation}
	\label{eq:gaussian_derivative}
	\beta^\gamma_h(u) = -\langle h,u\rangle_{H(\gamma)} \quad \textrm{for any} \quad h\in D(\gamma) = H(\gamma).
\end{equation}
\end{example}

\begin{remark}
Notice that in equation \eqref{eq:gaussian_derivative} the values of $u$ need not be in $H(\gamma)$.
In fact, the notation in \eqref{eq:gaussian_derivative} should be understood
as a measurable extension in $L^2(\gamma)$ 
(see Remark 8 in \cite{Lasanen1}). Also, if $\expec \gamma = u_0$, then the logarithmic derivative satisfies
\begin{equation*}
	\beta^\gamma_h(u) = - \langle h, u-u_0\rangle_{H(\gamma)}.
\end{equation*}
for any $h\in D(\gamma) = H(\gamma)$.
\end{remark}

Later we consider implications of general theory in the framework of convex probability measures.
A probability measure $\lambda$ on ${\mathcal B}(X)$ is called convex if, for all sets ${\mathcal A_1},{\mathcal A_2} \subset {\mathcal B}(X)$ and all $t \in [0,1]$, one has
\begin{equation}
\label{eq:def_convexity}
\lambda(t {\mathcal A_1} + (1-t) {\mathcal A_2}) \geq \lambda({\mathcal A_1})^t \lambda({\mathcal A_2})^{1-t}.
\end{equation}
Let us record here a well-known lemma regarding convex measures.
\begin{lemma}\cite[Prop. 4.3.8.]{Boga10}
Let $\lambda$ be a convex Radon measure on a locally convex space $X$ and let $V$ be a continuous convex function with $\exp(-V)\in L^1(\lambda)$. Then the measure $\nu = c \exp(-V) \cdot \lambda$ is convex,
where the number $c$ is a normalizing constant.
\end{lemma}
It directly follows that the posterior $\mu$ in \eqref{eq:bayes} given a Gaussian likelihood in \eqref{eq:gaussian_like} is convex.

In Section \ref{sec:besov} we construct an important example called the Besov prior by considering product measures. For the theory developed here, product measures provide a flexible framework as we will see from the following known results.
Suppose that $\mu_n$, $n\geq 1$ is Radon probability measures on a locally convex space $X_n$. Moreover, the dimension of $X_n$ is assumed to be one. Consider the Fomin differentiability of the
product measure $\mu = \otimes_{n=1}^\infty \mu_n$ on the space
$X = \prod_{n=1}^\infty X_n$.

\begin{theorem}\cite[Prop. 4.1.1.]{Boga10}
\label{thm:product_measures_conv}
Suppose that $\beta^{\mu_n}_{h_n}$ is the logarithmic derivative of $\mu_n$ in the direction $h_n\in X_n$. The following three claims are equivalent:
\begin{itemize}
	\item[(i)] $\mu$ is differentiable along $h = (h_j)_{j=1}^n \in X$,
	\item[(ii)] the series $\sum_{n=1}^\infty \beta_{h_n}^{\mu_n}$ converges in the norm of $L^1(\mu)$ and
	\item[(iii)] we have
	\begin{equation*}
		\sup_n \norm{\sum_{j=1}^n \beta^{\mu_j}_{h_j}}_{L^1(\mu)} < \infty.
	\end{equation*}
\end{itemize}
\end{theorem}

Let us introduce the following subspace
\begin{equation*}
	H(\mu) = \{h \in D(\mu) \; | \; \beta^\mu_h \in L^2(\mu)\} \subset D(\mu),
\end{equation*}
which has a natural Hilbert space structure \cite[Section 5]{Boga10}.
Surprisingly, for a large class of product measures $H(\mu)$ coincides with $D(\mu)$.
For the following corollary, see \cite[Cor. 2, page 43]{BS90} and Example 5.2.3 in \cite{Boga10}.

\begin{corollary}
\label{cor:prod_measures}
Suppose $m$ is a Borel probability measure on the real line such that 
\begin{equation*}
	\int_\R \frac{m'(t)^2}{m(t)} dt < \infty,
\end{equation*}
that is, $m$ has a finite Fisher information. If we have
$\mu_n(A) = m(A/a_n)$, where $a_n>0$, and $\mu = \otimes_{n=1}^\infty \mu_n$, then it follows that
\begin{equation*}
	D(\mu) = H(\mu) = \left\{h \in \R^\infty \; \left| \; \sum_{n=1}^\infty a_n^{-2} h_n^2 < \infty \right\}.\right.
\end{equation*}
\end{corollary}

\begin{proof}
From the definition we deduce that 
$\beta^{\mu_n}_{h} = a_n^{-1} \beta^{m}_{h}$ for any $h \in \R$.
Consequently, for any vector $h=(h_1,...,h_n,0,...) \in \R^\infty$ we have
\begin{equation*}
	\beta^\mu_h (u) = \frac{h_1}{a_1} \beta^{m}_{1}(u_1) + ... + \frac{h_n}{a_n} \beta^{m}_{1}(u_n)
\end{equation*}
where $u = (u_j)_{j=1}^\infty \in \R^\infty$, and
\begin{equation}
	\label{eq:beta_l2_bdedness}
	\norm{\beta^\mu_h}^2_{L^2(\mu)} = \left( \sum_{j=1}^n a_j^{-2} h_j^2 \right) \norm{\beta^{m}_{1}}^2_{L^2(\mu_1)}.
\end{equation}
Now if
$h=(h_j)_{j=1}^\infty$ such that $\sum_{j=1}^n a_j^{-2} h_j^2 < \infty$, then 
by Theorem \ref{thm:product_measures_conv} (iii) we have $h\in D(\mu)$. In addition, the series $\sum_{n=1}^N \beta_{h_n}^{\mu_n}$ converges to $\beta^\mu_h$ in $L^1(\mu)$ and by boundedness in equation \eqref{eq:beta_l2_bdedness} we have a subsequence that converges weakly in $L^2(\mu)$. Since the limit is unique, we have $\beta^\mu_h \in L^2(\mu)$ and, consequently, $h\in H(\mu)$.

Suppose now that $h \in D(\mu)$. Consider $\xi_n(u) = \beta^m_1(u_n)$ as 
a random variable $\xi_n : (X,{\mathcal B}(X), \mu) \to (\R, {\mathcal B}(\R))$. It follows that $\xi_n$ are independent, have zero mean and finite
second moment. In consequence, the characteristic functional $\phi_n$ of the random variable $\xi_n$ is twice differentiable at zero and there exists $\delta > 0$ such that
\begin{equation}
\label{eq:simple_char_estimate}
|1-\phi_1(t)|\geq \delta t^2
\end{equation}
in some neighbourhood of zero. Note carefully that $\phi_n(t) = \phi_1(\frac{h_n}{a_n}t)$.
Next, Theorem \ref{thm:product_measures_conv} yields the convergence of 
the series $\sum_{n=1}^\infty \beta^{\mu_n}_{h_n}$ in $L^1(\mu)$ and, similarly,
the mean convergence of the series $\sum_{n=1}^\infty \xi_n$. Together with the independence of $\xi_n$ we obtain the convergence of product $$\prod_{n=1}^\infty \phi_n\left(t\right)=\prod_{n=1}^\infty \phi_1\left(\frac{h_n}{a_n} t\right)< \infty.$$ Now it follows that the series $\sum_{n=1}^\infty |1-\phi_n(t)|$ must also be bounded
and by the estimate \eqref{eq:simple_char_estimate}
we have $\sum_{j=1}^n a_j^{-2} h_j^2 < \infty$. 
\end{proof}

\section{MAP Estimates from Small Balls and Translations}
\label{sec:onsager_and_map}

In this section we consider translated measures $\mu_h$, where 
$$\mu_h({\mathcal A}) = \mu({\mathcal A}-h)$$ for any ${\mathcal A} \in {\mathcal B}(X)$,
and work closely with the Radon--Nikodym derivative of $\mu_h$ with respect to $\mu$.
The measure $\mu$ is called quasi-invariant along the vector $h$ if $\mu_h$ is absolutely continuous with respect to $\mu$.
Also, recall that the support of $\mu$ is defined in the following way: if $x\in {\rm supp}(\mu)$ then every open neighbourhood of $x\in X$ has a non-zero measure. In what follows, we always assume without further mention that ${\rm supp}(\mu) = X$, i.e., $\mu$ has full support. In this manner we simplify the argumentation in e.g. next lemma (avoiding any division by zero).
Moreover, we make the following fundamental assumption on the measure $\mu$ appearing below:
\begin{itemize}
	\item[(A1)] there exists a separable Banach space $E \subset D(\mu)$ such that $E$ is topologically dense in $X$ and $\beta^\mu_h \in C(X)$ for any $h\in E$, i.e. $\beta^\mu_h$ has a continuous representative.
\end{itemize}
From this point on, whenever we write $h\in E$, the notation $\beta^\mu_h$ stands for the continuous representative.
As we illustrate below, assumption (A1) is satisfied for many typical prior distributions used in Bayesian inversion. 
Notice that (A1) also implies the topological density of $D(\mu)$ in $X$. 
The authors are not aware of general conditions for which the density of $D(\mu)$ would be guaranteed. The motivation
behind (A1) is given in the following simple lemma that connects the asymptotics of small ball probabilities to
the values of the continuous representative.
\begin{lemma}
\label{lem:cont_repres}
Assume that $\mu$ is quasi-invariant along the vector $h$. Denote the Radon--Nikodym derivative of $\mu_h$ with respect to $\mu$ by $r_h \in L^1(\mu)$. Suppose $r_h$ has a continuous representative $\tilde r_h \in C(X)$, i.e., $r_h - \tilde r_h = 0$ in $L^1(\mu)$. Then it holds that
\begin{equation*}
	\lim_{\epsilon \to 0} \frac{\mu_h(B_\epsilon(u))}{\mu(B_\epsilon(u))} = \tilde r_h(u)
\end{equation*}
for any $u \in X$.
\end{lemma}
\begin{proof}
By definition we have
\begin{equation*}
	\mu_h(B_\epsilon(u)) = \int_{B_\epsilon(u)} r_h(y) \mu(dy) = \int_{B_\epsilon(u)} \tilde r_h(y) \mu(dy)
\end{equation*}
for any $\epsilon>0$ and $u\in X$. It directly follows that
\begin{equation*}
	\min_{v \in B_\epsilon(u)} \tilde r_h(v) \leq \frac{\mu_h(B_\epsilon(u))}{\mu(B_\epsilon(u))} \leq \max_{v \in B_\epsilon(u)} \tilde r_h(v)
\end{equation*}
and the continuity yields the claim.
\end{proof}

The next proposition is a central tool that enables us to study non-Gaussian distributions from the
perspective of small ball probabilities.

\begin{proposition}\cite[Prop. 6.4.1]{Boga10}
\label{prop:onsager}
Suppose $\mu$ is a Radon measure on a locally convex space $X$ and is Fomin differentiable along a vector $h\in X$.
\correction{If it holds that $\exp(\epsilon |\beta^\mu_h(\cdot)|) \in L^1(\mu)$ for some $\epsilon>0$, then 
$\mu$ is quasi-invariant along $h$} and the Radon--Nikodym density $r_h$ of $\mu_h$ with respect to $\mu$
satisfies the equality
\begin{equation}
	\label{eq:general_onsager}
	r_h(u) = \exp\left(\int_0^1 \beta^\mu_h(u-sh)ds\right) \quad {\rm in }\; L^1(\mu).
\end{equation}
\end{proposition}

As a consequence of assumption (A1), $r_h$ has also a continuous representative if $h\in E$ and in this case
$r_{th}$ is differentiable with respect to $t\in\R$.
Similar to the logarithmic derivative, for any $h\in E$, the notation $r_h$ stands for this particular representative in what follows.
It is worth noticing that the assumption on the expectation of $\exp(\epsilon |\beta^\mu_h(\cdot)|)$ is important and
\correction{we have to assume it for the main results of this and the next section}:
\begin{itemize}
\item[(A2)] for any $h\in E$ there exists $\epsilon>0$ such that the prior probability measure $\lambda$ satisfies $\exp(\epsilon |\beta^\lambda_h(\cdot)|) \in L^1(\lambda)$.
\end{itemize}
Let us consider Proposition \ref{prop:onsager} for the Gaussian example.
\begin{example}
\label{ex:gaussian2}
Suppose $\gamma$ is chosen according to Example \ref{ex:gaussian1} and let us denote the Cameron--Martin space by $H = H(\gamma)$. First, due to the Fernique theorem \cite{Boga_gaussian} we have
\begin{equation*}
	\exp(|\beta^\gamma_h(\cdot)|) \leq \exp(\norm{\cdot}_X \norm{h}_E) \in L^1(\lambda)
\end{equation*}
and hence the assumption (A2)  
is satisfied for the Banach space $E=C X \subset X$.
Second, by equation \eqref{eq:gaussian_derivative} it follows that
\begin{equation*}
	\beta_h^\gamma(u-sh) = - \langle u-sh,h\rangle_H = -\langle u,h\rangle_H + s \norm{h}^2_H.
\end{equation*}
and, in consequence, 
\begin{equation*}
	r_{h}(u) = \exp\left(-\langle u,h\rangle_H + \frac 12 \norm{h}^2_H\right) \in L^1(\mu).
\end{equation*}
This is a classical result that can also be achieved by direct evaluations of the measures of small balls \cite{Lifshits}.
Notice that $r_h$ coincides with the Onsager--Machlup functional given in \cite{Dashti13} if $u\in D(\gamma) = H$.
\end{example}

Now we are ready to discuss the definition of a MAP estimate.
Let us first give the construction introduced in \cite{Dashti13}. Notice that 
we do not consider the questions of existence or uniqueness related to the MAP estimate.

\begin{definition}
\label{def:map}
Let
\begin{equation*}
	M^\epsilon = \sup_{u\in X} \mu(B_\epsilon(u)).
\end{equation*}
Any point $\hat u\in X$ satisfying 
$$\lim_{\epsilon\to 0}\frac{\mu(B_\epsilon(\hat u))}{M^\epsilon} = 1$$ is a MAP
estimate for the measure $\mu$.
\end{definition}

We remark that $\lim_{\epsilon\to 0} \left(\mu(B_\epsilon(u))/M^\epsilon\right) \leq 1$ holds for any $u\in X$.
Let us propose the following weaker characterization of the estimator.

\begin{definition}
\label{def:wmap}
We call a point $\hat u \in X$, $\hat u \in {\rm supp} (\mu)$, a weak MAP (wMAP) estimate if 
\begin{equation}
	\label{eq:wmap}
	r_h(\hat u) = \lim_{\epsilon \to 0} \frac{\mu(B_\epsilon(\hat u - h))}{\mu(B_\epsilon(\hat u))} \leq 1
\end{equation}
for all $h\in E$.
\end{definition}
The first equality in Definition \ref{def:wmap} is given by Lemma \ref{lem:cont_repres}. We give two direct implications that illustrate the nature of the weak MAP estimate.
\begin{lemma}
\label{lem:map_is_wmap}
Every MAP estimate $\hat u$ is a weak MAP estimate.
\end{lemma}

\begin{proof}
The claim is trivial since
\begin{equation*}
	r_h(\hat u) \leq \lim_{\epsilon\to 0} \frac{M^\epsilon}{\mu(B_\epsilon(\hat u))} = 1
\end{equation*}
for any $h\in E$.
\end{proof}

In the convex setting, it is natural that convex combinations of solutions (i.e. MAP or wMAP estimates) are solutions as well, which is confirmed by the following result: 
\begin{proposition}
\label{prop:MAP_set_is_convex}
Let $\mu$ be a convex measure. Moreover, let $\hat u, \hat v \in X$ and set $w = (1-t) \hat u + t \hat v$ for any $t \in [0,1]$. Then the following two claims holds
\begin{itemize}
	\item[(1)] If $\hat u$ and $\hat v$ are MAP estimates, then so is $w$.
	\item[(2)] If $\hat u$ and $\hat v$ are wMAP estimates and $\hat u - \hat v \in E$, then $w$ is a wMAP estimate.
\end{itemize}
\end{proposition}

\begin{proof}
The first claim follows directly by convexity (cf. equation \eqref{eq:def_convexity}) since
\begin{equation*}
	\lim_{\epsilon\to 0} \frac{\mu(B_\epsilon(w))}{M^\epsilon} \geq \lim_{\epsilon\to 0} \frac{\mu(B_\epsilon(\hat u))^t}{{(M^\epsilon)}^t}\frac{\mu(B_\epsilon(\hat v))^{1-t}}{{(M^\epsilon)}^{1-t}} = 1.
\end{equation*}
For the second property, 
notice that since $\hat u - \hat v \in E$, it must hold that $F_{\hat v - \hat u}(\hat u) = 1$.
Let us now write $g = \hat v - \hat u$ and $w = (1-t) \hat u + t \hat v = \hat u + t g$ for $t \in [0,1]$.
Let $h\in E$ be arbitrary. We have that
\begin{eqnarray*}
r_h(w) & = & \lim_{\epsilon\to 0} \frac{\mu(B_\epsilon(w-h))}{\mu(B_\epsilon(w))} \\
& = & \lim_{\epsilon\to 0} \left( \frac{\mu(B_\epsilon(\hat u + tg -h))}{\mu(B_\epsilon(\hat u))}
\cdot \frac{\mu(B_\epsilon(\hat u))}{\mu(B_\epsilon(\hat u + tg))} \right) \\
& \leq & \lim_{\epsilon\to 0} \left(\frac{\mu(B_\epsilon(\hat u))}{\mu(B_\epsilon(\hat u + g))} \right)^t
= r_{\hat v - \hat u}(\hat v)^t \leq 1,
\end{eqnarray*}
where we have used convexity and the fact that $\hat v$ satisfies equation \eqref{eq:wmap}. This yields the claim.
\end{proof}

In the next two theorems we describe sufficiency and necessity of a weak MAP estimate to be a zero point of $\beta^\mu_h$ for all $h\in E$. These results can be considered as the counterpart of sufficient and necessary optimality conditions in convex optimization.
\begin{theorem}
\label{thm:wmap_is_zero_point}
If $\hat u \in X$ is a weak MAP estimate of $\mu$, then $\beta^\mu_h(\hat u) = 0$ for all $h\in E$.
\end{theorem}

\begin{proof}
It follows from $r_h(\hat u) \leq 1$ and identity \eqref{eq:general_onsager} that
\begin{equation*}
	\int_0^t \beta^\mu_h(\hat u - sh) ds = \int_0^1 \beta^\mu_{th} (\hat u - s'\cdot th) ds' \leq 0
\end{equation*}
for all $h\in E$ and $t\in\R$. By continuity we then have $\beta^\mu_h(\hat u)\leq 0$. Now since $h,-h \in E \subset D(\mu)$ and by similar reasoning $\beta^\mu_{-h}(\hat u)\leq 0$, we must have 
\begin{equation*}
	0 \leq -\beta^\mu_{-h}(\hat u) = \beta^\mu_{h}(\hat u) \leq 0
\end{equation*}
and the claim follows.
\end{proof}

\begin{theorem}
\label{thm:zero_point_is_wmap}
Suppose that $\mu$ is convex and there exists $\tilde u \in X$ such that $\beta^\mu_h(\tilde u) = 0$
for all $h\in E$. Then $\tilde u$ is a weak MAP estimate.
\end{theorem}

\begin{proof}
Let us assume that $\tilde u$ is not a weak MAP estimate and, consequently, there exists $h\in E$
such that
\begin{equation*}
	r_h(\tilde u) = \lim_{\epsilon\to 0} \frac{\mu(B_\epsilon(\tilde u - h))}{\mu(B_\epsilon(\tilde u))}
	\geq 1+\delta.
\end{equation*}
Further, choose $\epsilon'>0$ be such that for any $\epsilon<\epsilon'$ we have
$$\frac{\mu(B_\epsilon(\tilde u - h))}{\mu(B_\epsilon(\tilde u))} \geq 1 + \frac \delta 2.$$
Recall the following equality for a sum of balls
\begin{equation*}
	tB_\epsilon(\tilde u - h) + (1-t) B_\epsilon(\tilde u)
	= B_\epsilon(t(\tilde u - h) + (1-t) \tilde u) = B_\epsilon(\tilde u - th)
\end{equation*}
for $0\leq t \leq 1$. Now by convexity we have
\begin{equation*}
	\mu(B_\epsilon(\tilde u - th)) \geq \mu(B_\epsilon(\tilde u - h))^t \mu(B_\epsilon(\tilde u))^{1-t}
	\geq \left(1+ \frac \delta 2\right)^{t} \mu(B_\epsilon(\tilde u)).
\end{equation*}
This yields an inequality
\begin{equation*}
	r_{th}(\tilde u) \geq \left(1+\frac \delta 2\right)^{t}
\end{equation*}
and, finally, we conclude that
\begin{equation*}
	\beta^\mu_h(\tilde u) = \partial_t r_{th}(\tilde u) |_{t=0} \geq \left.\partial_t \left(1+\frac \delta 2\right)^{t}\right|_{t=0} = \ln\left(1+\frac \delta 2\right)
\end{equation*}
since $r_0(\tilde u) = (1+\delta/2)^0 = 1.$
This yields a contradiction with our assumption $\beta^\mu_h(\tilde u) = 0$.
\end{proof}

\section{Variational characterization and the posterior}
\label{sec:variational}

In the following we discuss the further variational characterization of MAP estimates in the inverse problems setting \eqref{eq:bayes} and the associated posterior distribution. 
Let us first discuss differentiation by parts in Fomin calculus.

\begin{proposition}[\cite{Boga10} Prop. 3.3.12.]
Let $\lambda$ be a measure differentiable along $h$ and let $f$ be a bounded measurable function possessing
a uniformly bounded partial derivative $\partial_h f $. Then, the measure $\mu = f \cdot \lambda$
is differentiable along $h$ as well and one has
\begin{equation}
	\label{eq:partial_diff}
	d_h \mu = \partial_h f \cdot \lambda + f \cdot d_h \lambda.
\end{equation}
\end{proposition}

The proposition above can be directly applied to posterior distribution given in equation \eqref{eq:bayes}.
From this point on, we require that the likelihood distribution is Gaussian.

\begin{theorem}
\label{thm:variational_map}
Let $\mu$ and $\lambda$ be the posterior and prior probability distribution in equation \eqref{eq:bayes}, respectively, and the likelihood is given by \eqref{eq:gaussian_like}. Moreover, suppose that Assumption (A1) holds.
Then we have $D(\lambda) \subset D(\mu)$ and 
	the posterior distribution $\mu$ has a logarithmic derivative $\beta_h^\mu \in L^1(\mu)$ such that
	\begin{equation}
		\beta_h^\mu(u) = -\bra Au-m, Ah\cet_{\R^M} + \beta_h^\lambda(u)
	\end{equation}
	for any $h \in D(\mu)$. Moreover, if Assumption (A2) holds, then also $\exp(|\epsilon \beta^\mu_h|) \in L^1(\mu)$ for some $\epsilon>0$.
\end{theorem}

\begin{proof}
	Let $h\in D(\mu)$ and denote $f(u) = \exp(-\frac 12 |Au-m|^2)$ for $u\in X$. Since $A: X \to \R^M$ is bounded, the function $t\mapsto f(u+th)$ is everywhere differentiable. Moreover, we have $f\in L^1(d_h \lambda)$ and $\p_h f\in L^1(\lambda)$.
	By equation \eqref{eq:partial_diff} it follows that
	\begin{eqnarray*}
		d_h \mu & = & f \cdot d_h \lambda + \p_h f \cdot \lambda \\
		& = & \left(\beta^\lambda_h(\cdot) - \bra A\cdot-m, Ah\cet_{\R^M}\right) f \cdot \lambda.
	\end{eqnarray*}
	\correction{This yields the first part of the claim.} Now assume that $\exp(\epsilon |\beta^\lambda_h|) \in L^1(\lambda)$ holds for some $\epsilon>0$. We have
	\begin{multline*}
		\norm{\exp(\epsilon |\beta^\mu_h(\cdot)|)}_{L^1(\mu)} \\
		\leq C \int_X \exp(\epsilon(C_1|Au-m|+|\beta^\lambda_h(u)|)) \exp\left(-\frac 12 |Au-m|^2\right) \lambda(du) \\
		\leq \widetilde C \int_X \exp\left(-\frac 12 (|Au-m|-C_2)^2\right) \exp(\epsilon|\beta^\lambda_h(u)|) \lambda (du) \\
		\leq \widetilde C\norm{\exp(\epsilon|\beta^\lambda_h(\cdot)|)}_{L^1(\lambda)},
	\end{multline*}
	where $C,\widetilde C,C_1,C_2>0$ are suitable constants.
\end{proof}

The variational characterization of the wMAP estimate is a direct consequence:

\begin{corollary}
\label{cor:variational_map_J}
Let us assume that $\mu$ and $\lambda$ are as in Theorem \ref{thm:variational_map}. Moreover, we assume that the prior distribution $\lambda$ is a convex measure and there is an (unbounded) convex functional $J : X \to [0,\infty]$, which is Frechet differentiable everywhere in its domain $D(J)$ and $J'(u)$ has a bounded extension $J'(u) : E \to \R$ such that
\begin{equation*}
	\beta^\lambda_h(u) = -J'(u)h
\end{equation*}
for any $h\in E$ and any $u\in X$. Further, we require that $E\cap D(J)$ is topologically dense in $X$. Then a point $\hat u$ is a weak MAP estimate if and only if $\hat u \in \argmin_{u\in X} F(u)$ where
\begin{equation}
	\label{eq:posterior_F}
	F(u) = \frac 12 |Au-m|^2+ J(u).
\end{equation}
\end{corollary}

\begin{proof}
By Theorems \ref{thm:wmap_is_zero_point} and \ref{thm:zero_point_is_wmap}, a point $\hat u$ is a weak MAP estimate if and only if satisfies
\begin{equation*}
	-\beta^\mu_h(\hat u) = \bra Au-m, Ah\cet_{\R^M} + J'(u)h = 0.
\end{equation*}
Recall that by assumption (A1) the subspace $E$ is topologically dense and hence the claim holds.
\end{proof}

\begin{remark}
Suppose that there exists a MAP estimate for the posterior distribution $\mu$ and the corresponding functional $F$ in equation \eqref{eq:posterior_F} has a unique minimum. Then Lemma \ref{lem:map_is_wmap} directly yields that $\mu$ has 
a unique (strong) MAP estimate given by the minimizer of $F$. 
\end{remark}

\begin{example}
\label{ex:gaussian3}
Let $\gamma_W$ be a zero-mean Gaussian measure in Example \ref{ex:gaussian1} defined on $X := H^{-t}(\T^d)$ with $t>s/d$ such that $T = (I-\Delta)^{-t/2}$. A random variable with probability distribution $\gamma_W$ is called \emph{white noise} due to the property
\begin{equation*}
	\int_X \bra u, \phi\cet_{{\mathcal D}'} \bra u, \psi\cet_{{\mathcal D}'} \gamma_W(du)
	= \bra \phi, \psi \cet_{L^2(\T^d)}
\end{equation*}
for any $\phi,\psi\in C^\infty(\T^d)$, where $\bra \cdot, \cdot \cet_{{\mathcal D}'}$ stands for the 
distribution duality. From Example \ref{ex:gaussian1} we notice $D(\gamma_W) = H(\gamma_W) = L^2(\T^d)$ and conclude that $J(u) = \frac 12 \norm{u}_{L^2}^2$ satisfies assumptions in Corollary \ref{cor:variational_map_J} and the weak MAP estimate can be found by minimizing
\begin{equation}
	F(u) = \frac 12 |Au-m|^2+ \frac 12 \norm{u}_{L^2}^2.
\end{equation}
\end{example}

In earlier work \cite{BL14} by the authors, the MAP estimate was characterized by a Bayes cost method
using the Bregman distance
\begin{equation}
	\label{eq:bregman}
	D_J(u,v) = J(u) - J(v) - J'(v) (u-v).
\end{equation}
This approach is not directly possible in an infinite-dimensional setting since integrals of type
$\int_X J(v) \mu(dv)$ are not well-defined. In fact, the domain of $J$ has typically zero measure in terms of the posterior. For example, in case of Gaussian measure, it corresponds to the Cameron--Martin space. In non-Gaussian problems $D(J)$ and $D(\mu)$ do not always coincide, e.g. for the Besov prior in Section \ref{sec:besov} we have $D(J) \subset D(\mu)$.

We avoid the problem described above by considering homogeneous Bregman distance
\begin{equation}
	\label{eq:hom_bregman}
	\D_J(u,\cdot) = J(u) + \beta^\lambda_{u}(\cdot)
\end{equation}
in $L^1(\lambda)$ for any $u \in D(\lambda) \cap D(J)$. In other words, we neglect part of the Bregman distance that, from the perspective of minimizing $u$ in $\int_X D_J(u,v) \mu(dv)$, is "constant". \correction{Note that for one-homogeneous functionals $J$, i.e., $J(tu) = |t| J(u)$ for all $t\in\R$, homogeneous Bregman distance coincides exactly with the Bregman distance.}

In order to achieve a quadratic formulation below, we assume that the prior $\lambda$ 
has a finite second moment:
\begin{itemize}
	\item[(A3)] $\int_X \norm{u}^2_X \lambda(du) < \infty$.
\end{itemize} 
Also, recall that the conditional mean (CM) estimate is defined by
\begin{equation}
	\label{eq:cm}
	\cm = \int_X u \mu(du).
\end{equation}
The next lemma is available in more generality in \cite{Buldygin}. However, for convenience we record
a simplification here.
\begin{lemma}
\label{lem:var_cm}
Let $L : X \to Y$ be linear, bounded and invertible, where $Y$ is a separable Hilbert space. Then for any $\beta > 0$ we have
\begin{equation}
	\cm = \argmin_{u\in X} \int_X \left(\frac \lambda 2 |Au-Av|^2 + \frac \beta 2 \norm{Lu-Lv}_Y^2 \right) \mu(dv)
\end{equation}
\end{lemma}

\begin{proof}
We have
\begin{align*}
& \int_X \left(\frac 1 2 |Au-Av|^2 + \frac \beta 2 \norm{Lu-Lv}_Y^2 \right) \mu(dv) = \\
& \qquad 
\int_X \left(\frac 1 2 |Au_{CM}-Av|^2 + \frac \beta 2 \norm{Lu_{CM}-Lv}_Y^2 \right) \mu(dv) + \\
& \qquad 
\int_X \left(\frac 1 2 |Au-Au_{CM}|^2 + \frac \beta 2 \norm{Lu-Lu_{CM}}_Y^2 \right) \mu(dv) +
\\
& \qquad 
\int_X  \left\langle A^* A(u-u_{CM} ) + \beta  L^* L(u-u_{CM} ), u_{CM}-v  \right\rangle_{X^*\times X} \mu(dv)
\end{align*}
for any $u\in X$.
The first term on the right-hand side is the cost for $u_{CM}$, the second term is nonnegative and vanishes only for $u=u_{CM}$, and the last one vanishes due to linearity and the definition of the CM estimate. Thus, $u=u_{CM}$ is the unique minimizer.
\end{proof}

In the next theorem we finally arrive to a characterization of the weak MAP estimate with Bregman distance and Bayes cost. As a consequence, we can also give a result stating that in terms of Bayes cost with respect to $\D_J$ a weak MAP estimate is more optimal than the CM estimate.
\begin{theorem}
\label{thm:cost}
Assume that $\mu$ and $\lambda$ are as in Corollary \ref{cor:variational_map_J} and Assumptions (A1)-(A3)
are satisfied. Then the following two claims hold:
\begin{itemize}
\item[(1)]
The vector $\hat u \in X$ is a weak MAP estimate if and only if it minimizes
the functional
\begin{equation}
	\label{eq:var_map}
	G(u) = \int_X \left(\frac 12 |Au-Av|^2 + \D_J(u,v)\right) \mu(dv), \quad
	u \in D(J)\cap D(\lambda),
\end{equation}
where $G(u) = \infty$ for $u\in X \setminus D(J)\cap D(\lambda)$.
\item[(2)]Moreover, let $\map$ and $\cm$ denote a wMAP and the CM estimate. We have that
\begin{equation*}
	\int_X \D_J(\map,u) \mu(du) \leq \int_X \D_J(\cm,u) \mu(du).
\end{equation*}
\end{itemize}
\end{theorem}

\begin{proof}
The first claim follows by rewriting $G(u)$, $u\in D(J)\cap D(\lambda)$, as 
\begin{eqnarray*}
G(u) & = & \frac 12 |Au|^2 + J(u) + \int_X \left(-\bra A^* Av,u\cet_{X^*\times X}+\beta^\lambda_u(v)\right) \mu(dv) + C \\
& = & \frac 12 |Au-m|^2 + J(u) + C,
\end{eqnarray*}
where $C$ is a constant and we have used the fact that the integral term $\int_X \beta^\mu_h(v) \mu(dv)$ vanishes.

For the second claim, we can assume $\cm \in D(\lambda)$ since otherwise the statement is trivial.
Following \cite{BL14} we obtain 
\begin{multline*}
	\int_X \left(|A\map - Au)|^2 + \D_J(\map, u)\right) \mu(du) \\
	\leq \int_X \left( |A\cm - Au)|^2 + \D_J(\cm, u)\right) \mu(du) \\
	\leq \int_X \left(|A\map - Au)|^2 + \D_J(\cm, u)\right) \mu(du) \\
	+ \beta \int_X \left(\norm{L\map-Lu}_Y^2-\norm{L\cm-Lu}_Y^2\right)\mu(du).
\end{multline*}
where we have utilized Lemma \ref{lem:var_cm} and the first claim.
Since $\beta>0$ is arbitrary, we can consider $\beta\to 0$ and obtain the result.
\end{proof}

\section{Example 1: Besov prior}
\label{sec:besov}

In this Section we consider the Besov space prior studied in \cite{LSS09}.
Suppose that functions $\{\psi_\ell\}_{\ell=1}^\infty$ form
an orthonormal wavelet basis 
for $L^2(\T^d)$, where we have utilized a global indexing.
We can characterize the periodic Besov space $B^s_{pq}(\T^d)$ using the given basis
in the following way: the series
\begin{equation}
	f(x) = \sum_{\ell=1}^\infty c_\ell \psi_\ell(x)
\end{equation}
belongs to $B^s_{pq}(\T^d)$ if and only if
\begin{equation}
	\label{eq:orig_norm}
	2^{js} 2^{j(\frac 12 -\frac 1p)} \left( \sum_{\ell=2^j}^{2^{j+1}-1} |c_\ell|^p\right)^{1/p} \in \ell^q(\N).
\end{equation}
We assume that the basis is $r$-regular for $r$ is large enough in order to provide basis for a Besov space with smoothness $s$ \cite{Daubechies}.

In the following we are concerned with the special case $p=q$ and write $B^s_{p} = B^s_{pp}$. It is well-known that
an equivalent norm to \eqref{eq:orig_norm} is given by
\begin{equation}
	\norm{\sum_{\ell=1}^\infty c_\ell\psi_\ell}_{B^s_{p}(\T^d)}
	= \left(\sum_{\ell=1}^\infty \ell^{p(\frac sd+\frac 12)-1} |c_\ell|^p \right)^{1/p}.
\end{equation}
We now follow the construction in \cite{LSS09} to define a Besov prior
using the wavelet basis. Notice also the work in \cite{Dashti_besov,Ramlau} on non-linear problems and
Besov prior on the full space $\R^d$, respectively.
\begin{definition}
\label{def:besov_prior}
Let $1 \leq p<\infty$ and let $(X_\ell)_{\ell=1}^\infty$ be independent identically
distributed real-valued random variables with the probability density function
\begin{equation}
	\pi_X(x) = \sigma_p \exp(-|x|^p) \quad {\rm with } \quad
	\sigma_p = \left(\int_\R \exp(-|x|^p)dx\right)^{-1}.
\end{equation}
Let $U$ be the random function
\begin{equation*}
	U(x) = \sum_{\ell=1}^\infty \ell^{-\frac sd-\frac 12 + \frac 1p} X_\ell \psi_\ell(x),
	\quad x \in \T^d.
\end{equation*}
Then we say that $U$ is distributed according to a $B^s_p$ prior.
\end{definition}

\begin{lemma}\cite[Lemma 2]{LSS09}
\label{lem:besov_lss}
Let $U$ be as in Definition \ref{def:besov_prior} and let $t<s-\frac dp$. Then it holds that
\begin{itemize}
	\item[(i)] $\norm{U}_{B^t_{p}} < \infty$, \quad almost surely, and
	\item[(ii)] $\expec \exp(\frac 12 \norm{U}_{B^t_{p}}^p ) < \infty$.
\end{itemize}
\end{lemma}

\correction{Let us denote by $\rho_\ell$ the probability measure of random variable $\ell^{-s/d-1/2+1/p}X_\ell$ on $\R$.
Next we consider the product measure of coefficients $(\ell^{-s/d-1/2+1/p}X_\ell)_\ell$ in $(\R^\infty,{\mathcal B}(\R^\infty))$ and denote it by $\lambda=\prod_{\ell=1}^\infty \rho_\ell$.}
Notice carefully that here $\R^\infty$ does not have a Banach structure.
By Lemma \ref{lem:besov_lss} we could as well consider versions of $\lambda$ in a subspace of $\R^\infty$ that correspond to a Besov space $B^t_p(\T^d)$ with some fixed $t<s-\frac dp$.
However, the parameter $t$ does not play any role in the analysis below and
the set of differentiability is not affected by this choice.
In addition, the results on product measures in Section \ref{sec:prelim} become directly available.

In the following we assume that $p>1$ in order to have $\pi_X$ differentiable. Also, the prior would not satisfy assumption (A1) for $p=1$. It remains for future work to generalize the given results to this important case.
We point out that $\lambda$ is convex due to \cite[Prop. 4.3.3]{Boga10} and clearly has a full support.

\begin{theorem}
The set of differentiability is given by
$D(\lambda) = B^{s+(\frac 12 - \frac 1p)d}_{2}(\T^d)$ for $p>1$ and for any $h = \sum_{\ell=1}^\infty h_\ell \phi_\ell \in D(\lambda)$ we have
\begin{equation*}
	\beta_h^\lambda(u) = \lim_{N\to\infty} \sum_{\ell=1}^N \beta_{h_\ell}^\lambda(u_\ell)
	\quad {\rm in } \; L^1(\lambda),
\end{equation*}
where
\begin{equation*}
\beta_{h_\ell}^\lambda(u_\ell) = h_\ell \left(-p \sign (u_\ell) \ell^{p(\frac sd+\frac 12)-1} |u_\ell|^{p-1}\right)
\end{equation*}
is the logarithmic derivative of $\rho_\ell$.
\end{theorem}

\begin{proof}
The probability distribution $\pi_X$ on $\R$ has finite Fisher information since
\begin{equation}
	\int_\R \frac{\pi_X'(t)^2}{\pi_X(t)} dt = p^2 \int_\R |t|^{2(p-1)} \pi(t) dt = p \sigma_p \Gamma\left(2-\frac 1p\right) < \infty
\end{equation}
for any $p>1$, where $\Gamma$ is the Gamma function. Clearly, it holds that
\begin{equation}
	\rho_\ell({\mathcal A}) = \rho_1(c_\ell{\mathcal A}),
\end{equation}
for any ${\mathcal A} \in {\mathcal B}(\R)$, where $c_\ell = \ell^{s/d+1/2-1/p}$. 
By Corollary \ref{cor:prod_measures} we have
\begin{equation*}
	D(\lambda) = \left\{y\in\R^\infty \; | \; \sum_{\ell=1}^\infty c_\ell^2 y_\ell^2<\infty\right\} = \left\{y \in \R^\infty \; | \; \sum_{\ell=1}^\infty \ell^{\frac{2s}{d} + 1-\frac 2p} y_\ell^2<\infty\right\}. 
\end{equation*}
and, consequently, $D(\lambda) = B^{s'}_{2}(\T^d)$ for $s'=s + (1/2-1/p)d$.
\end{proof}

We turn our attention to assumption (A1) and to subspace $E \subset D(\lambda)$.

\begin{lemma}
\label{lem:besov_beta_ineq}
For any $h\in B_{p}^{ps-(p-1)t}(\T^d)$ the logarithmic derivative $\beta^\lambda_h$ has the upper bound 
\begin{equation*}
	|\beta^\lambda_h(u)| \leq C \norm{u}_{B^t_{p}}^{p-1} \norm{h}_{B_{p}^{ps-(p-1)t}}.
\end{equation*}
$\lambda$-almost surely.
\end{lemma}

\begin{proof}
The logarithmic derivative satisfies
\begin{eqnarray}
	\label{eq:besov_ineq_aux1}
	\left|\sum_{\ell=1}^N \beta_{h_\ell}^\lambda(u_\ell)\right| & \leq & C \sum_{\ell=1}^\infty \ell^{p(\frac sd+\frac 12)-1} |u_\ell|^{p-1} |h_\ell| \nonumber \\
	& \leq & C \left(\sum_{\ell=1}^\infty \ell^{r_1 \frac{p}{p-1}} |u_\ell|^p\right)^{1-1/p} 
	\left(\sum_{\ell=1}^\infty \ell^{r_2p} |h_\ell|^p \right)^{1/p},
\end{eqnarray}
for some constant $C>0$ and any $N$. Above, we have applied H\"{o}lder inequality with $\left(\frac{p}{p-1},p\right)$-exponents and chosen $r_1$ and $r_2$ in such a way that
\begin{equation*}
	r_1 \frac{p}{p-1} = p\left(\frac td +\frac 12\right)-1, \quad \textrm{that is,} \quad
	r_1 = (p-1)\left(\frac td +\frac 12\right) - 1 + \frac 1p
\end{equation*}
and 
\begin{equation*}
	r_1+r_2 = p\left(\frac sd + \frac 12\right) -1.
\end{equation*}
It follows that
\begin{equation*}
	r_2 p  = 
	\frac pd (ps-(p-1)t)+ \frac p2 -1.
\end{equation*}
Such parametrization applied to inequality \eqref{eq:besov_ineq_aux1} yields the claim.
\end{proof}

It turns out in the following that the choice $E= B^{ps-(p-1)t}_p(\T^d)$ satisfies the required assumptions.

\begin{lemma}
\label{lem:besov_exp_in_L1}
We have $\exp(|\beta^\lambda_h|) \in L^1(\lambda)$ for any $h\in E$.
\end{lemma}

\begin{proof}
We write $A_1 = \{u\in B^t_p \; | \; \norm{u}_{B^t_p} \geq 2 C \norm{h}_{B^{ps-(p-1)t}_p}\}$ and
$A_2 = B^t_p \setminus A_1$. Now it follows by Lemmas \ref{lem:besov_beta_ineq} and \ref{lem:besov_lss} that
\begin{multline*}
\int_{B^t_p(\T^d)} \exp(|\beta^\lambda_h(u)|) \lambda(du) \\ \leq \int_{A_1} \exp\left(\frac 12 \norm{u}^p_{B^t_p}\right) \lambda(du) 
+ \int_{A_2} \exp(C \cdot (2C)^{p-1}) \lambda(du) < \infty
\end{multline*}
for any $h\in E$.
\end{proof}

Given the bound in Lemma \ref{lem:besov_beta_ineq} we can pointwise define
$\tilde \beta^\lambda_h(u) = \sum_{\ell=1}^\infty \beta_{h_\ell}^\lambda(u_\ell)$
when $h\in E$.

\begin{proposition}
\label{prop:besov_has_cont_repres}
It holds that $\tilde \beta_h^\lambda \in C(B^t_p(\T^d))$ for any $h\in B^{ps-(p-1)t}_p(\T^d)$ and $1<p\leq 2$.
\end{proposition}

\begin{proof}
Suppose $u^n$ converges to $u$ in $B^t_p(\T^d)$ when $n$ increases and denote $u^n = \sum_{\ell=1}^\infty u^n_\ell \psi_\ell$.
Now we have 
\begin{multline*}
	\left|\tilde \beta^\lambda_h(u) - \tilde \beta^\lambda_h(u^n)\right| \\
	= C\sum_{\ell=1}^\infty \ell^{p(s/d+1/2)-1} \left|\sign(u_\ell) |u_\ell|^{p-1}-\sign(u^n_\ell) |u^n_\ell|^{p-1}\right|  |h_\ell|
\end{multline*}
and since the following elementary inequality
\begin{equation*}
	\left|\sign(a) |a|^{p-1} - \sign(b) |b|^{p-1}\right| \leq 2 |a-b|^{p-1}
\end{equation*}
holds for arbitrary $a,b\in \R$ with $1\leq p \leq 2$, we must have
$\lim_{n\to\infty} \beta^\lambda_h(u^n) = \beta^\lambda_h(u)$ by similar H\"{o}lder argument as in Lemma \ref{lem:besov_beta_ineq}.
\end{proof}

From Lemma \ref{lem:besov_exp_in_L1} and Proposition \ref{prop:besov_has_cont_repres} we obtain that the probability distribution $\lambda$ of a $B_p^s$ prior with $1<p\leq 2$ satisfies assumption (A1) and (A2) with $X=B_p^t(\T^d)$ for any $t<s-d/p$ and $E = B_p^{ps-(p-1)t}(\T^d)$. Moreover, we find that the functional 
\begin{equation*}
	J(u) = \sum_{\ell=1}^\infty \ell^{ps/d+p/2-1} |u_\ell|^p = \norm{u}_{B^s_{p}}^p 
\end{equation*}
satisfies assumptions in Corollary \ref{cor:variational_map_J}
and we have $\beta^\lambda_h(u) = -J'(u)h$ for all $h\in E$.
As an interesting remark, notice that $D(J) = B^s_p(\T^d) \subset D(\lambda) = B^{s+(\frac 12 - \frac 1p)d}(\T^d)$ by a well-known embedding theorem \cite{SchmeisserTriebel}.

Since $E = B^{ps-(p-1)t}_p(\T^d)$ is separable and dense in $X = B^t_p(\T^d)$ with some fixed $t<s-\frac dp$,
the weak MAP estimate to problem \eqref{eq:inv_problem} under assumptions used in Corollary \ref{cor:variational_map_J} is obtained by minimizing the functional
\begin{equation*}
	F_{Besov}(u) = \frac 12 |Au-m|^2 + \norm{u}_{B^s_{p}}^p.
\end{equation*}
Notice that $F_{Besov}$ has unique minimum.

\section{Example 2: Hierarchical prior}
\label{sec:hierarchical}

Consider a situation where a Gaussian distribution seems to provide a good prior model for the unknown with the only drawback that there is high uncertainty regarding the mean value. In this case, the mean value becomes part of the Bayesian inference problem and the prior then takes the form of what is often called a hierarchical model \cite{KS}.

In the following, let $\nu$ be a zero-mean Gaussian measure on a separable Hilbert space $X$ with a covariance operator $C$. Recall that notation $\nu_y(\cdot) = \nu(\cdot - y)$ stands for the translated measure
and $\nu_y$ and $\nu$ are absolutely continuous if and only if $y\in H(\nu)$, i.e., $y$ belongs to the Cameron--Martin space of $\nu$. Suppose now that $e \in C X \subset H(\nu)$ and let a parameter $t\in\R$ be distributed according to probability density $\rho \in C^1(\R)$ such that $\rho>0$ everywhere. We define our hierarchical prior $\lambda$ by the equality
\begin{equation*}
	\lambda({\mathcal A}\times {\mathcal T})
	= \int_{\mathcal T} \nu_{te}({\mathcal A}) \rho(t)dt
\end{equation*}
for any ${\mathcal A} \in {\mathcal B}(X)$ and ${\mathcal T} \in {\mathcal B}(\R)$. Below, we write $\nu_t = \nu_{te}$ for convenience. In addition, we assume that there exists $\epsilon>0$ such that
\begin{equation}
	\label{eq:hierarchical_assumption}
\exp(\epsilon (|t| + |\beta^\rho(t)|)) \in L^1(\rho \cdot dt).
\end{equation}
Note that e.g. any Gaussian distribution on $\R$ satisfies equation \eqref{eq:hierarchical_assumption}.

The convexity of $\lambda$ follows if $\rho$ is a convex distribution. Namely, any finite-dimensional projection has
a Gaussian probability density or a density of type
\begin{equation}
	\label{eq:hierarchical_finite_density}
	\pi(u,t) \propto \exp\left(-\frac 12 \norm{\widetilde C^{-1/2}(u-te)}^2_{\R^n}\right) \rho(m)
\end{equation}
where $\widetilde C \in \R^{n\times n}$ is the projected covariance matrix. Since the squared norm in equation \eqref{eq:hierarchical_finite_density} is convex with respect to $(u,t) \in \R^{n+1}$, the convexity of $\lambda$ is obtained by \cite[Prop. 4.3.3.]{Boga10}. In addition, $\lambda$ clearly has a full support.
\begin{theorem}
We have
$D(\lambda) = D(\nu) \times \R$ and
\begin{equation}
\label{eq:beta_hierarchical}
\beta^{\lambda}_{(h,1)} (u,t) = \beta^{\nu_t}_h(u) - \beta^{\nu_t}_e(u) + \beta^\rho(t)
\end{equation}
in $L^1(\lambda)$.
\end{theorem}

\begin{proof}
Let us consider the Fomin derivative for the Borel set $\tilde {\mathcal A} = {\mathcal A}\times {\mathcal T} \in {\mathcal B}(X\times \R)$, where ${\mathcal A} \in {\mathcal B}(X)$ and ${\mathcal T}\in {\mathcal B}(\R)$. By straightforward manipulation we can decompose the differential as follows
\begin{eqnarray*}
I^\epsilon & = & \frac 1\epsilon \left(\lambda_{-\epsilon (h,1)}(\tilde {\mathcal A}) - \lambda(\tilde {\mathcal A})\right) \\
& = & \frac 1 \epsilon \left( \int_{\mathcal T+\epsilon} \nu_{t}({\mathcal A}+\epsilon h) \rho(t)dt - \int_{\mathcal T} \nu_{t}({\mathcal A}) \rho(t)dt \right) \\
& = & I^\epsilon_1 + I^\epsilon_2 + I^\epsilon_3,
\end{eqnarray*}
where
\begin{eqnarray*}
I^\epsilon_1 & = & \frac 1 \epsilon \int_{\mathcal T+\epsilon} \left(\nu_{t}({\mathcal A}+\epsilon h)- \nu_{t}({\mathcal A})\right) \rho(t)dt \\
I^\epsilon_2 & = & \frac 1 \epsilon \int_{\mathcal T} \left(\nu_{t+\epsilon}({\mathcal A})- \nu_t({\mathcal A})\right) \rho(t+\epsilon)dt \quad {\rm and} \\
I^\epsilon_3 & = & \frac 1 \epsilon \int_{\mathcal T} \nu_t({\mathcal A})\left(\rho(t+\epsilon) - \rho(t)\right) dt.
\end{eqnarray*}

Since $e \in CX \subset D(\nu)$ and $d\nu_t({\mathcal A})$ is uniformly bounded with respect to $t\in\R$, we can apply the Lebesgue dominated convergence theorem to each term and obtain
\begin{eqnarray}
	\label{eq:hier_limit_diff}
	d_{(h,1)}\lambda(\tilde {\mathcal A}) & = & \lim_{\epsilon\to 0} (I^\epsilon_1+I^\epsilon_2+I^\epsilon_3)\nonumber \\
	& = & \int_{\mathcal T} \left( d_h \nu_t ({\mathcal A}) \rho(t) 
	+ d_{-e} \nu_t({\mathcal A}) \rho(t) 
	+ \nu_t({\mathcal A}) \rho'(t) \right)dt\nonumber \\
	& = & \int_{\tilde {\mathcal A}} \left(\beta^{\nu_t}_h(u) - \beta^{\nu_t}_{e}(u) + \beta^\rho(t) \right) \lambda(du\times dt).
\end{eqnarray}
We see that the limit in \eqref{eq:hier_limit_diff} is finite and due to countably additivity of $\lambda$ the equality generalizes for any 
set belonging to
\begin{equation*}
	{\mathbb A} = \left\{\bigcup_{j=1}^\infty {\mathcal A}_j \times {\mathcal T}_j \; \bigg| \; {\mathcal A}_j \in {\mathcal B}(X), {\mathcal T}_j \in {\mathcal B}(\R),  \{{\mathcal A}_j \times {\mathcal T}_j\}_{j=1}^\infty \; \textrm{are disjoint}\right\}.
\end{equation*}
Now it is easy to see that ${\mathbb A}$ forms a Dynkin system. On the other hand, the measurable sets of form $\tilde {\mathcal A} = {\mathcal A} \times {\mathcal T}$ form a $\pi$-system which belongs to ${\mathbb A}$. By Dynkin's $\pi$-$\lambda$ theorem one obtains that ${\mathcal A} = {\mathcal B}(X\times \R)$ and since $\beta^{\nu_t}_h(u), \beta^{\nu_t}_{e}(u) \in L^1(\lambda)$, the claim follows.
\end{proof}

\begin{lemma}
The hierarchical prior $\lambda$ satisfies following properties:
\begin{itemize}
\item[(1)] for any $h\in C X$ the logarithmic derivative $\beta^\lambda_{(h,1)}$ has a continuous representative, i.e., in $E = C X \otimes \R$ and
\item[(2)] $\exp(\epsilon |\beta^{\lambda}_{(h,1)}(\cdot)|) \in L^1(\lambda)$ for some $\epsilon>0$.
\end{itemize}
\end{lemma}

\begin{proof}
The first claim follows trivially from equation \eqref{eq:beta_hierarchical}.

In order to prove the second claim, consider the following bound
\begin{equation*}
	\beta^\lambda_{(h,1)}(u,t) \leq C_1 \norm{u}_X + C_2 |t| + |\beta^\rho(t)|
\end{equation*}
where the constants $C_1, C_2> 0$ depend on the variable $e$. Using the Cameron--Martin formula we can write
\begin{multline*}
I_\lambda = \int_\R \int_X \exp(\epsilon|\beta^\lambda_{(h,1)}(u,t)|) \nu_t(du) \rho(t) dt \\
\leq \int_\R \int_X \exp(\epsilon(C_1 \norm{u}_X + C_2 |t| + |\beta^\rho(t)|))  \\
\times \exp\left(t\bra e, u\cet_{H(\nu)}-\frac{t^2}2 \norm{e}^2_{H(\nu)}\right) \nu(du) \rho(t) dt
\end{multline*}
for any $\epsilon>0$. The power of second exponential term can be reformulated as
\begin{equation}
\label{eq:hierarchical_proof_aux}
t\bra e, u\cet_{H(\nu)}-\frac{t^2}2 \norm{e}^2_{H(\nu)} =
\frac{\bra e, u\cet_{H(\nu)}^2}{2\norm{e}_{H(\nu)}^2} - \left(\frac{\bra e, u\cet_{H(\nu)}}{\sqrt 2\norm{e}_{H(\nu)}}- \frac t{\sqrt 2} \norm{e}_{H(\nu)}\right)^2.
\end{equation}
Since the last term in \eqref{eq:hierarchical_proof_aux} is always negative, the integral $I_\lambda$ can be bounded by following decomposition:
\begin{equation*}
	I_\lambda \leq \int_X \exp\left(\epsilon(C_1 \norm{u}_X + C_3 \norm{u}_X^2)\right) \nu(dx) \times
	\int_\R \exp\left(\epsilon(C_2 |t| + |\beta^\rho(t)|)\right) dt.
\end{equation*}
The integrals above are bounded for some $\epsilon$ due to the Fernique theorem \cite{Boga_gaussian} and assumption in equation \eqref{eq:hierarchical_assumption}, respectively.
\end{proof}

For simplicity, let us assume that $\rho$ is a normal distribution. In consequence, the logarithmic derivative of the posterior $\lambda$ has the form
\begin{equation*}
	\beta^\lambda_{(h,1)}(u,t) = -\bra u-e,u-te\cet_{H(\nu)} - t
\end{equation*}
and, thus, the functional
\begin{equation*}
	J(u,t) = \frac 12 \norm{u-te}^2_{H(\nu)} + \frac 12 t^2
\end{equation*}
satisfies assumption in Corollary \ref{cor:variational_map_J}. Moreover, the wMAP estimate to problem \eqref{eq:inv_problem} is obtained by minimizing
functional
\begin{equation*}
	F_{Hierarchical}(u,t) = \frac 12 |Au-m|^2 + \frac 12 \norm{u-te}^2_{H(\nu)} + \frac 12 |t|^2
\end{equation*}
for $(u,t) \in X \times \R$. In this case, the problem has a unique wMAP estimate.

\section{Discussion and Conclusions}

In this article, we have examined the role of the MAP estimate for an infinite-dimensional Bayesian inverse problems.
The topic has been scarcely studied in the literature mainly because of the difficulty to 
define a MAP estimate for such a problem. Most importantly, it is difficult to connect a topological definition as in \cite{Hegland} to a variational problem. The first breakthrough in this regard was achieved in \cite{Dashti13}
where Gaussian distributions for non-linear problems were considered.
Here, we introduced a novel concept of a weak MAP estimate, which allows a variational study of the MAP estimate in a rather general framework and might be the basis for further results on non-Gaussian priors. As a first step we consider the wMAP estimate in the Bayes cost method utilizing Bregman distances. Similar work for finite-dimensional problems was done in \cite{BL14}.

We recognize that our work leaves some fascinating questions open. Besides the obvious question of infinite-dimensional measurements, we list here three directions of future work. Firstly, our analysis leaves out the case of $p=1$ in the Besov example. This case is particularly interesting due to its role as a sparsity prior. The drawback is that the logarithmic derivative does not have continuous representative (not even in finite dimensions), i.e., assumption (A1) is not satisfied. A generalization would probably need to move from logarithmic derivatives of measures to some new definition taking into account convexity, similar to the step from differentiation to subdifferentials in convex optimization. 

Secondly, under what conditions 
is a weak MAP estimate also a (strong) MAP estimate in sense of Definition \ref{def:map}?
Recall that according to Lemma \ref{lem:map_is_wmap} every MAP estimate is a weak MAP estimate.
Hence, an intriguing challenge is to construct a posterior distribution for which there exist only weak MAP estimates, or conversely, to show that these two concepts coincide.

Thirdly, the role of the set of differentiability $D(\mu)$ is interesting as well. For example, when does the (weak) MAP estimate belong to $D(\mu)$? Notice that for the Besov prior, the wMAP estimate belongs to even a smaller subspace $D(J) \subset D(\mu)$. In what generality does this phenomenon appear? 

Finally, the ability to formulate the MAP estimate via variational approaches can open up new frontiers of research
in infinite-dimensional Bayesian inverse problems as the study of non-Gaussian problems becomes more feasible.
Further, we believe that results represented here provide new insight to discretization invariance \cite{LSS09} and,
consequently, can have direct impact in practical applications. \\

\noindent {\bf Acknowledgements} \\

\noindent The work of TH was supported by the ERC-2010 Advanced
Grant, 267700 - InvProb (Inverse Problems) and by the Academy of Finland via the grant 275177.
The work of MB was partially supported by ERC via Grant EU FP 7 - ERC Consolidator Grant 615216 LifeInverse and by the German Science Foundation DFG via BU 2327/6-1 and EXC 1003 Cells in Motion Cluster of Excellence, M\"unster, Germany.
The authors thank Sari Lasanen (Oulu) and Steffen Dereich (M\"unster) for thoughtful comments regarding this work.

\bibliographystyle{abbrv}

\bibliography{references}

\end{document}